\documentclass[11pt]{article}
\usepackage[all]{xy}
\usepackage{amssymb,amsmath,amsthm, mathrsfs}
\usepackage{graphicx}
\usepackage{setspace}
\usepackage{fullpage}

\newtheorem{theorem}{Theorem}[section]
\newtheorem{lemma}[theorem]{Lemma}

\newtheorem{definition}[theorem]{Definition}
\newtheorem{example}[theorem]{Example}

\newtheorem{proposition}[theorem]{Proposition}
\newtheorem{corollary}[theorem]{Corollary}
\newenvironment{acknowledgement}[1][Acknowledgements]
{\begin{trivlist} \item[\hskip \labelsep {\bfseries #1}]}
{\end{trivlist}}

\newtheorem{remark}[theorem]{Remark}
 \DeclareMathSymbol{\N}{\mathbin}{AMSb}{"4E}

\DeclareMathSymbol{\Z}{\mathbin}{AMSb}{"5A}
\DeclareMathSymbol{\R}{\mathbin}{AMSb}{"52}
\DeclareMathSymbol{\Q}{\mathbin}{AMSb}{"51}
\DeclareMathSymbol{\I}{\mathbin}{AMSb}{"49}
\DeclareMathSymbol{\C}{\mathbin}{AMSb}{"43}

\def\A{{\mathbb A}}

\def\X{{\mathbb X}}
\def\Y{{\mathbb Y}}
\def\CO{{\mathscr{C}}}

\numberwithin{equation}{section}

\begin{document}

\title{Families of algebraic varieties parametrized by topological spaces}

\author{Jyh-Haur Teh\\
{\footnotesize Department of Mathematics, National Tsing Hua
University of Taiwan}}

\date{}
\maketitle

\footnotetext[1]{E-mail address: jyhhaur@math.nthu.edu.tw}

\begin{abstract}
We study families of algebraic varieties parametrized by topological
spaces and generalize some classical results such as Hilbert
Nullstellensatz and primary decomposition of commutative rings. We
show that there is an equivalence between the category of bivariant
coherent sheaves and the category of sheaves of finitely generated
modules.
\end{abstract}

\section{Introduction}
Algebraic varieties are defined by the zero sets of some polynomials
over some fields. When the field is $\C$, with the natural analytic
topology, algebraic varieties have the homotopy types of
CW-complexes. Topological considerations may unveil some very
interesting properties of algebraic varieties. For example, the
homotopy groups of the groups of algebraic cycles of varieties which
are called the Lawson homology groups, are important invariants of
varieties (see \cite{Law, F}). Therefore we are interested in
studying varieties parametrized by topological spaces, especially by
the sphere $S^n$. Since the standard unit sphere $S^n\subset
\R^{n+1}$ is a real algebraic variety, we may consider polynomials
with coefficients of real rational functions on $S^n$. It is clear
that the zero sets defined by such polynomials may not vary
continuous. But out of a set of measure zero, they do vary
continuously, namely, the Hilbert polynomial of each fibre is the
same except over a set of measure zero. To study such families, we
develop a foundation theory in this paper. We generalize several
classical results such as Hilbert Nullstellensatz and the primary
decomposition of commutative rings.

In the following, we give a brief overview of each section. In
section 2, we define algebraic varieties over topological spaces and
their coordinate rings. In section 3, we study Groebner ideals in
polynomial rings with continuous coefficients, introduce measure
topology, condensed spaces, condensed ideals, almost Groebner
ideals, prove a Hilbert Nullstellensatz and an almost primary
decomposition theorem. In section 4, we develop a theory of
bivariant coherent sheaves over semi-topological algebraic sets, and
show that there is an equivalence of the category of almost sheaves
of $K[\X]$-modules and the category of bivariant quasi-coherent
sheaves on $\X$.

\begin{acknowledgement}
The author thanks Li Li for his very helpful comments, and Eric
Friedlander, Christian Haesemeyer, Mark Levine for their
encouragement. Special thanks to my wife Yu-Wen Kao for her
understanding. He also thanks the Taiwan National Center for
Theoretical Sciences (Hsinchu) for providing a wonderful working
environment.
\end{acknowledgement}

\section{Affine algebraic sets}
\begin{definition}
Given a ringed space $(S, \mathcal{O}_S)$. For each open set
$U\subseteq S$, we define
$$(\mathcal{O}_S[x_1, ..., x_n])(U):=\mathcal{O}_S(U)[x_1, ...,
x_n]$$ where $\mathcal{O}_S(U)[x_1, ..., x_n]$ is the ring of
polynomials of $n$ variables with coefficients in
$\mathcal{O}_S(U)$. In other words, it is the collection of all
polynomials of the form
$$f(x)=\sum_{|J|\leq k}a_Jx^J$$ for some $k\geq 0$ where $J\in \Z^n_{\geq
0}$, $|J|=i_1+\cdots +i_n$ if $J=(i_1, ..., i_n)$, and all $a_J\in
\mathcal{O}_S(U)$. The degree of $f$ is defined to be the maximal
$|J|$ if $a_J$ is not identically zero. The restriction maps of
$\mathcal{O}_S[x_1, ..., x_n]$ are induced by restricting
coefficients.
\end{definition}

Note that $\mathcal{O}_S[x_1, ..., x_n]$ is in general not a sheaf.
After we introduce the concept of a condensed space and almost
sheaf, we will see that $\mathcal{O}_S[x_1, ..., x_n]$ is actually
an almost sheaf.

In this paper, we denote by $S$ for a topological space and
$\mathscr{C}$ the sheaf of germs of continuous complex-valued
functions.

\begin{definition}
Given a topological space $(S, \mathscr{C})$. The affine $n$-space
over $S$ is the set $\A^n_S:=S\times \C^n$. We denote by
$\pi:S\times \C^n \rightarrow S$ the natural projection.
\end{definition}

\begin{definition}
Given a topological space $(S, \CO)$. Let $J\subseteq
\mathscr{C}[x_1, ..., x_n](S)$ be an ideal of the global sections.

Let
$$\X:=\{(s, x)|f_s(x)=0, \mbox{ for all } f\in J\}\subseteq \A^n_S$$
where $f_s$ is the restriction of $f$ to the fibre $\pi^{-1}(s)$ of
$s$. $\X$ is called an affine algebraic set over $S$. It induces a
presheaf of sets on $S$ defined by
$$\X(U):=\X\cap \pi^{-1}(U)$$ where $U$ is an open subset of $S$.

For $U$ an open set of $S$, we write
$$J|_U:=\{f|_U|f\in J\}$$

By abuse of notation, we also write $J$ for the ideal sub-presheaf
of $\mathscr{C}[x_1, ..., x_n]$ induced by $J$ which is defined by
$$J(U)=\mbox{ ideal of } \mathscr{C}[x_1, ..., x_n](U) \mbox{
generated by } J|_U$$

The vanishing presheaf of $\X$, denoted by $I(\X)$, is the presheaf
defined by
$$I(\X)(U):=\{f\in \CO[x_1, ..., x_n](U)|f(\X(U))=0\}$$

The radical presheaf of $J$ is defined by
$$\sqrt{J}(U):=\{f\in \CO[x_1, ..., x_n](U)|f^n\in J(U)
\mbox{ for some } n\in \N\}$$

The fibrewise radical presheaf of $J$ is defined by
$$FRad(J)(U):=\{f\in \CO[x_1, ..., x_n](U)|f_s\in \sqrt{J(U)|_s}, s\in
U\}$$ where $J(U)|_s=\{g_s|g\in J(U)\}$.

We write
$$V(J)(U):=\{(s, x)|f_s(x)=0, f\in J(U)\}$$ and

$$\X_s:=\X\cap \pi^{-1}(s)$$ for $s\in S$.
\end{definition}

\begin{definition}\label{coordinate almost sheaf}
Let $(S, \mathscr{C})$ be a topological space and $\X\subseteq
\A^n_S$ be an affine algebraic set over $S$. For an open set
$U\subseteq S$, the presheaf which assigns an open set $U$ the ring
$$K[\X](U):=\CO[x_1, ..., x_n](U)/I(\X)(U)$$ is called the coordinate presheaf of $\X$.
\end{definition}

The following example shows that we can not expect an analog of
Hilbert's Nullstellensatz to hold for arbitrary ideals.

\begin{example}
Let $S=[-1, 1]$ be the closed interval from $-1$ to $1$. Let
$I=<e^{-\frac{1}{s}}>\subseteq \CO[x](S)$, $\X=V(I)=0\times \C$,
$f_s(x)=s$. Then $f\in \mathscr{C}[x](S)$ and $f(\X)=0$. Since
$\lim_{s\to 0}|\frac{s^n}{e^{-\frac{1}{s}}}|=\infty$ for any $n\in
\N$, therefore $f^n\notin I$ for any $n\in \N$. This shows that
$f\in FRad(I)$ but $f\notin Rad(I)$.
\end{example}

\begin{proposition}
Given a topological space $(S, \mathscr{C})$ and $I\subseteq
\CO[x_1, ..., x_n]$ be an ideal. Then $IV(I)=FRad(I)$.
\label{Hilbert's Nullstellensatz}
\end{proposition}

\begin{proof}
If $f\in IV(I)(U)$, then $f(V(I(U)))=0$. Therefore
$f_s(V(I(U))|_s)=0$ for all $s\in U$. By Hilbert's Nullstellensatz,
we have $f_s\in \sqrt{I(U)|_s}$ for $s\in S$. Hence $f\in
FRad(I)(U)$. Another direction is trivial.
\end{proof}

\begin{definition}
Let $S$ be a compact topological space. If $a\in \CO(U)$ for
$U\subseteq S$, we write $|a(s)|$ for the Euclidean norm of the
complex number $a(s)$. For $f, g\in \CO[x_1, ..., x_n](S)$, we
define a norm by
$$||f-g||:=\sup_{s\in S}\{|a_{\alpha}(s)-b_{\alpha}(s)|\ | \alpha\in
A\}$$ where $f=\sum_{\alpha\in A}a_{\alpha}(s)x^{\alpha},
g=\sum_{\alpha\in A}b_{\alpha}(s)x^{\alpha}$, and $A$ is a finite
subset of $\Z^n_{\geq 0}$.
\end{definition}

\begin{proposition}
Let $S$ be a compact topological space. Let $I\subseteq
\mathscr{C}[x_1, ..., x_n](S)$ be an ideal. If $I$ is a fibrewise
radical ideal, i.e., $FRad(I)=I$, then $I$ is closed in the norm
defined above.
\end{proposition}

\begin{proof}
Suppose that $f_i$ approaches $f$ where $f_i\in I$ and $f\in
\CO[x_1, ..., x_n](S)$. The convergence of $f_i$ in the norm to $f$
implies that $f_i$ converges to $f$ pointwise. Since $f_i(V(I))=0$
for each $i$, we have $f(V(I))=0$. Hence $f\in IV(I)=FRad(I)=I$.
\end{proof}

In order to prove a version of Hilbert's Nullstellensatz to relate
$IV(J)$ and $\sqrt{J}$, we need to have a bound $e$ such that if
$f_{s_0}\in \sqrt{I|_{s_0}}$, then $f_{s}^e\in I|_s$ for $s$ in a
neighborhood of $s_0$. For the convenience of the reader, we recall
a result of Brownawell (see corollary of Theorem 1 of
\cite{Brown87}).

\begin{theorem}(Brownawell)
If $Q, P_1, ..., P_m\in \C[x_1, ..., x_n]$ have degrees $\leq D$ and
if $Q$ vanishes on all the common zeros of $P_1, ..., P_m$ in
$\C^n$, then there are $e\in \N$ and polynomials $B_1, ..., B_m\in
\C[x_1, ..., x_n]$ with
$$e\leq e':=(\mu+1)(n+2)(D+1)^{\mu+1}$$
$$deg B_iP_i\leq eD+e'$$
where $\mu=min\{m, n\}$ such that
$$Q^e=B_1P_1+\cdots +B_mP_m.$$\label{Brownawell}
\end{theorem}

With the help of topology, we can say more about the Hilbert's
Nullstellensatz for a principal ideal.

\begin{proposition}
Let $S$ be a compact topological space. Let $g\in \CO[x_1, ...,
x_n](S)$ be such that $g_s$ is not is not a zero polynomial for
$s\in S$. Then $FRad(<g>)=\sqrt{<g>}$ as presheaves.\label{principal
Hilbert Nullstellensatz}
\end{proposition}

\begin{proof}
If $h\in FRad(<g>)(U)$, then by Theorem \ref{Brownawell}, there is
$N$ such that $h^N_s\in <g_s>$ for all $s\in U$. Hence we can find
$c_s\in \mathscr{C}[x_1, ..., x_n](U)$ such that $h^N_s=c_sg_s$. We
need to show that $c$ is continuous in $s$. Fix $s_0\in S$. Then
$h^N_s-h^N_{s_0}=(c_s-c_{s_0})g_s+c_{s_0}(g_s-g_{s_0})$. Since the
degrees of $\{c_s|s\in U\}$ are bounded, the limit $\lim_{s \to
s_0}c_s$ taking under the norm defined above exists. So we have
$\lim_{s\to s_0}(c_s-c_{s_0})g_{s_0}=0$. Since $g_{s_0}$ is not the
zero polynomial which implies that $\lim_{s\to s_0}(c_s-c_{s_0})=0$
and we have $\lim_{s \to s_0}c_s=c_{s_0}$. The continuity of $c$
implies that $h\in \sqrt{<g>}(U)$.
\end{proof}

\section{Groebner ideals}
In order to generalize Hilbert Nullstellensatz, we work on nice
topological spaces and ideals with nice properties. We develop the
theory of Groebner ideals of the ring of continuous functions over
condensed spaces. The results of this chapter play fundamental role
in our theory. The main reference for Groebner ideals is \cite{CLS}.

\begin{definition}
Fix a monomial ordering (see \cite[Definition 2.2.1]{CLS}) of
$\C[x_1, ..., x_n]$. Let $S$ be a topological space. For $h\in
\CO[x_1, ..., x_n](S)$, write
$$h_s(x)=a_1(s)x^{\alpha_1}+a_2(s)x^{\alpha_2}+\cdots
+a_k(s)x^{\alpha_k}$$ where $x^{\alpha_1}>x^{\alpha_2}>\cdots
>x^{\alpha_k}$ according to the monomial ordering where $\alpha_i=(a_1, ...,
a_n)$ is a multi-index. As in \cite{CLS}, we denote by
$LT(h)=a_1(s)x^{\alpha_1}$ the leading term of $h$, $LC(h)=a_1(s)$
the leading coefficient of $h$, $LM(h)=x^{\alpha_1}$ the leading
monomial of $h$ and $multideg(h)=\alpha_1$ the multi degree of $h$.
For $f\in \CO[x_1, ..., x_n](S)$, we say that the leading term
$LT(f)$ of $f$ is nonvanishing on $S$ if $LC(f)$ is nonvanishing on
$S$.
\end{definition}

\begin{definition}
Let $S$ be a topological space. Let $f, g\in \mathscr{C}[x_1, ...,
x_n](S)$ be two nonzero polynomials. Let $\alpha=(\alpha_1, ...,
\alpha_n)=multideg(f), \beta=(\beta_1, ..., \beta_n)=multideg(g)$
and $\gamma_i=max\{\alpha_i, \beta_i\}$, $\gamma=(\gamma_1, ...,
\gamma_n)$. If $LT(f), LT(g)$ are nonvanishing on $S$, the
$S$-polynomial $S(f, g)$ of $f$ and $g$ is
$$S(f, g)=\frac{x^{\gamma}}{LT(f)}f-\frac{x^{\gamma}}{LT(g)}g$$
\end{definition}

\begin{definition}
Let $S$ be a topological space. Let $G=(f_1, ..., f_r)$ where each
$f_i\in \mathscr{C}[x_1, ..., x_n](S)$ and $LT(f_i)$ is nonvanishing
on $S$. For any $f\in \mathscr{C}[x_1, ..., x_n](S)$, we write
$\overline{f}^F$ for the remainder on division of $f$ by the ordered
$r$-tuple $F$ (see \cite[Definition 2.6.3]{CLS}).
\end{definition}

\begin{definition}
Let $S$ be a topological space. Let $f_1, ..., f_k\in
\mathscr{C}[x_1, ..., x_n](S)$ where all $LT(f_i)'s$ are
nonvanishing on $S$. The set $\{f_1, ..., f_k\}$ is a Buchbergerable
set if the following recursive procedure return `TRUE'.

REPEAT\\
        $G':=G$ \\
 \ \ \             FOR \mbox{ each pair }  $\{p, q\}, p\neq q \mbox{ in } G'$
 \ \ \             DO \\
 \ \ \             $g:=\overline{S(p, q)}^{G'}$ \\
 \ \ \             IF $g\neq 0$ AND $LT(g)$  nonvanishing THEN $G:=G\cup \{g\}$\\
 \ \ \             IF $g\neq 0$ AND $LT(g)$  vanishing  THEN RETURN FALSE AND STOP\\

UNTIL $G=G'$ RETURN TRUE AND STOP\\
\end{definition}

\begin{definition}
We say that a finite generating set $\{g_1, ..., g_k\}$ of an ideal
$I\subseteq \mathscr{C}[x_1, ..., x_n](S)$ is a Groebner basis of
$I$ if
\begin{enumerate}
\item
the ideal generated by the leading terms $<LT(g_1), ...,
LT(g_k)>$ of the generating set is same as the ideal generated
by the leading terms $<LT(I)>$ of the elements of the ideals $I$
over $S$;

\item
the leading terms of $g_i's$ are nonvanishing on $S$.
\end{enumerate}

An ideal with a Groebner basis is called a Groebner ideal.
\end{definition}

Throughout this paper, we fix a monomial ordering of $\C[x_1, ...,
x_n]$.

\begin{lemma}
Let $(S, \mathscr{C})$ be a topological space. Let $h\in \CO[x_1,
..., x_n](S)$ and $I=<g_1, ...., g_k>$ be an ideal of $\CO[x_1, ...,
x_n](S)$ generated by $G=\{g_1, ..., g_k\}$ such that the leading
terms of all $g_i's$ are nonvanishing over $S$.

\begin{enumerate}
\item (Continuous division algorithm)  We have
$h=\sum^k_{i=1}a_ig_i+r$ where all $a_i, r\in \CO[x_1, ...,
x_n](S)$, and either $r=0$ or $r$ is a linear combination, with
coefficients in $\mathscr{C}(S)$, of monomials, none of which is
divisible by any of $LT(g_1), ..., LT(g_k)$. (see \cite[Theorem
2.3.3]{CLS}).

\item
(Buchberger's criterion) The generating set $G$ is a Groebner
basis of $I$ if and only if the remainder on division of $S(g_i,
g_j)$ by $G$ is zero for any $i\neq j$.

\item (Continuous Buchberger's algorithm)
If $G$ is Buchbergerable, a Groebner basis for $I$ can be
constructed from the set $G=\{g_1, ..., g_k\}$ by Buchberger's
algorithm (see \cite[Theorem 2.7.2]{CLS}).

\item Let $\{h_1, ..., h_l\}$ be a Groebner basis of $I$. Then
$h=\sum_{i=1}^la_ih_i+r$ where all $a_i, r\in \CO[x_1, ...,
x_n](S)$ and $h\in I$ if and only if $r=0$.
\end{enumerate} \label{algorithms}
\end{lemma}

\begin{proof}
(1) Since the leading terms of all $g_i's$ are nonvanishing, they
can be inverted, hence we can do the division algorithm as usual.
(2) The proof is similar to the proof of \cite[Theorem 2.6.6]{CLS}.
(3) The definition of a Buchberger set is to make the proof of
Buchberger algorithm works. (4) The reason that $h$ has such
expression follows from division algorithm and the rest follows from
properties of a Groebner basis.
\end{proof}

\begin{example}
Let $S=[0, 1]$, $f(x, y)=x^2+sy^2$, $g(x, y)=x+y$. We use the
lexicographic order such that $x^ay^b>x^cy^d$ if $a>c$ or $a=c$ and
$b>d$. Let $G=\{f, g\}$. Then
$$S(f, g)=-xy+sy^2, h(x, y)=\overline{S(f, g)}^{G}=(s+1)y^2$$
and the procedure ends. So $\{f, g, h\}$ is a Groebner basis for the
ideal generated by $f, g$.
\end{example}

\begin{corollary}
Let $S$ be a topological space. If $\{g_1, ..., g_k\}$ is a Groebner
basis of an ideal $I\subseteq \mathscr{C}[x_1, ..., x_n](S)$, then
$<LT(g_1)_s, ..., LT(g_k)_s>=<LT(I|_s)>$ which means that $\{g_{1,
s}, ..., g_{k, s}\}$ is a Groebner basis for $I|_s$ for each $s\in
S$.
\end{corollary}

\begin{proof}
By (2) of the Lemma above, for any $i\neq j$, we have $S(g_i,
g_j)=\sum_{l=1}^kh_lg_l$ for some $h_l\in \mathscr{C}[x_1, ...,
x_n](S)$. Then for $s\in S$, $S(g_{i, s}, g_{j, s})=S(g_i,
g_j)|_s=\sum_{l=1}^kh_{l, s}g_{l, s}$. Since $I|_s=<g_{1, s}, ...,
g_{k, s}>$, by \cite[Theorem 2.6.6]{CLS}, $\{g_{1, s}, ..., g_{k,
s}\}$ is a Groebner basis of $I|_s$.
\end{proof}

\begin{proposition}
Suppose that $S$ is a path-connected topological space and
$I\subseteq \mathscr{C}[x_1, ..., x_n](S)$ is a Groebner ideal. Then
the Hilbert polynomial $HP_{I|_s}$ is same for $s\in S$. In
particular, $\mbox{dim } V(I|_s)=\mbox{dim } V(I|_{s'})$ for any $s,
s'\in S$.
\end{proposition}

\begin{proof}
Let $f_1, ..., f_k$ be a Groebner basis for $I$. Let $\gamma:[0, 1]
\rightarrow S$ be a path such that $\gamma(0)=s$, $\gamma(1)=s'$.
Then $LC(f_{i, \gamma(t)})$ is nonvanishing for any $t\in [0, 1]$.
Thus $LM(f_{i, s})=LM(f_{i, s'})$ and therefore $LT(f_{i,
s'})=\frac{LC(f_{i, s'})}{LC(f_{i, s})}LT(f_{i, s})$ where
$\frac{LC(f_{i, s'})}{LC(f_{i, s})}\in \mathscr{C}[x_1, ...,
x_n](S)$. So $<LT(I|_s)>=<LT(f_{1, s}), ..., LT(f_{k, s})>=<LT(f_{1,
s'}), ..., LT(f_{k, s'})>=<LT(I|_{s'})>$. The result follows from
\cite[Proposition 9.3.4]{CLS} which proves that the monomial ideal
$<LT(I|_s)>$ has the same Hilbert polynomial as $I|_s$.
\end{proof}

\subsection{Almost Groebner ideals} For properties about real
algebraic sets, we refer the reader to \cite{BCR}.

\begin{definition}(Condensed spaces)
Let $S\subseteq \R^m$ be a topological subspace. A polynomial
function on $S$ is a continuous function $f:S \rightarrow \C$ such
that $f(x_1, ..., x_m)=\sum^n_{i=0, |\alpha|=i}c_{\alpha}x^{\alpha}$
for some $n\in \N$, $c_{\alpha}\in \C$. Let $\mathscr{P}(S)$ be the
ring of all polynomials on $S$. A real algebraic subset of $S$ is
the zero loci of finitely many polynomials on $S$. The space $S$ is
said to be condensed if for any real algebraic subset $V\subset S$,
$V\neq S$, $V$ is nowhere dense under the Euclidean topology of $S$.
Let
$$\mathscr{R}(S)=\{\frac{f}{g}|f, g\in \mathscr{P}(S), 0\notin
g(S)\}$$ An element of $\mathscr{R}(S)$ is called a rational function on $S$ and
an element of $\mathscr{R}(S)[x_1, ..., x_n]$ is called a condensed
element of $S$.
\end{definition}

\begin{remark}
The topology of real algebraic sets are much more complicated than
complex algebraic sets. Easy to see that all spheres $S^n$,
connected intervals on the real line, unit disks are all condensed.
See \cite[Example 3.1.2]{BCR} for real algebraic varieties that are
not condensed.
\end{remark}

\begin{definition}(Measure topology)
Suppose that $S$ is a condensed space. A closed set of $S$ in the
measure topology is defined to be a set of the form
$$\cap^{\infty}_{i=1}\cup^{\infty}_{j=1}V_{ij}$$ where $V_{ij}$ are
some real algebraic subsets of $S$.
\end{definition}
It is easy to check that the finite union and countably intersection
of sets of this form are again a set of this form.

\begin{definition}
Let $S$ be a condensed space. Let $\Sigma$ be the collection of all
closed sets of $S$ under the measure topology except the whole $S$,
then $(S, \Sigma)$ is a sigma space. We will always consider a
condensed space under measure topology, and as a sigma space in this
way.
\end{definition}

\begin{proposition}
If $S$ is a condensed space, then any open set $U\subseteq S$ is
condensed.
\end{proposition}

\begin{proof}
 Let
$V\subseteq U$ be a real algebraic subset of $U$ which is the zero
loci of some polynomials $f_1, ..., f_k$ on $U$. We may consider
$f_1, ..., f_k$ as polynomials on $S$ and let $\widetilde{V}$ be the
zero loci of them in $S$. If $V$ is not nowhere dense in $U$, then
there exists an open set $W\subseteq U$ such that $W\cap V$ is dense
in $W$. Since $W$ is also an open subset of $S$, and $W\cap
\widetilde{V}\supset W\cap V$ is dense in $W$, $S$ is not condensed
which is a contradiction.
\end{proof}

\begin{definition}(Condensed ideals)
Suppose that $S\subseteq \R^N$ is a condensed space. An ideal
$I\subseteq \mathscr{C}[x_1, ..., x_n](S)$ is said to be condensed
if $I$ is generated by some elements of $\mathscr{R}[x_1, ...,
x_n](S)$. The presheaf induced by a condensed ideal is called a
condensed presheaf.
\end{definition}

\begin{corollary}(Finite generation of condensed ideals)
If $I\subseteq \mathscr{C}[x_1, ..., x_n](S)$ is a condensed ideal
where $S$ is a condensed space, then $I$ is finitely
generated.\label{finite generation condensed ideals}
\end{corollary}

\begin{proof}
Let $\Sigma$ be the collection of all condensed elements of $I$.
Then $\Sigma$ is an ideal of $\mathscr{P}(S)[x_1, ..., x_n]$. Since
$\mathscr{P}(S)$ is Noetherian, $\mathscr{P}(S)[x_1, ..., x_n]$ is
Noetherian, hence $\Sigma$ is finitely generated. Therefore $I$ is
finitely generated.
\end{proof}

\begin{definition}(Condensed basis)
Let $I\subseteq\mathscr{C}[x_1, ..., x_n](S)$ be a condensed ideal
where $S$ is a condensed space. If $I$ is generated by some
condensed elements $f_1, f_2, ..., f_m$, then $\{f_1, f_2, ...,
f_m\}$ is called a condensed basis of $I$.
\end{definition}

\begin{definition}(Sigma spaces)
Let $S$ be a topological space and $\Sigma$ be a collection of
closed subsets of $S$. If any countably union of elements in
$\Sigma$ is again in $\Sigma$, then the pair $(S, \Sigma)$ is called
a sigma space and elements of $\Sigma$ are called sets of measure
zero.
\end{definition}

\begin{definition}(Almost sheaves)
Let $S$ be a sigma space and $\mathscr{F}$ be a presheaf on $S$. An
element $\sigma\in \mathscr{F}(U)$ is said to be almost zero,
denoted by $a=_a 0$, if there is a set $A\subseteq S$ of measure
zero such that $\sigma|_{U-A}=0$. Two elements $\sigma_1\in
\mathscr{F}(U), \sigma_2\in \mathscr{F}(V)$ are almost equal,
denoted by $\sigma_1=_a \sigma_2$, if there is a set $A$ of measure
zero such that $\sigma_1|_{U-A}=\sigma_2|_{V-A}$. For $\alpha\in
\mathscr{F}(V)$, we write $\alpha \in_a \mathscr{F}(U)$ if there is
a set $A$ of measure zero such that $\alpha|_{V-A}\in
\mathscr{F}(U-A)$.

We say that $\mathscr{F}$ is an almost sheaf if it satisfies the
following two properties:
\begin{enumerate}
\item (Almost Uniqueness) For $\sigma\in \mathscr{F}(U)$, if there is an open covering $U=\cup_i U_i$
such that the restriction $\sigma|_{U_i}=_a 0$ for any $i$, then
$\sigma=_a 0$;

\item (Almost extension property) for any open set $U$ of $S$ and for any open covering
$\{U_i\}$ of $U$, if $\sigma_i\in \mathscr{F}(U_i)$ and
$\sigma_i|_{U_i\cap U_j}=_a\sigma_j|_{U_j\cap U_i}$ for any $i,
j$, then there is $\sigma \in_a \mathscr{F}(U)$ such that
$\sigma|_{U_i}=_a \sigma_i$ for any $i$.
\end{enumerate}
\end{definition}

\begin{definition}
Let $S$ be a condensed space with the measure topology and
$I\subseteq \mathscr{C}[x_1, ..., x_n](S)$ be a finitely generated
ideal.
\begin{enumerate}
\item We say that $I$ is almost Groebner if there exist a subset $A\subseteq S$
of measure zero such that $I(S-A)$ is a Groebner ideal of
$\mathscr{C}[x_1, ..., x_n](S-A)$. The presheaf induced by $I$
is called an almost Groebner presheaf.

\item We say that $I$ is almost Buchbergerable if there exist a subset
$A \subseteq S$ of measure zero such that there exist a
generator set $\{f_1, ..., f_k\}$ of $I(S-A)$ which is a
Buchbergerable set. The presheaf induced by $I$ is called an
almost Buchbergerable presheaf.
\end{enumerate}
\end{definition}

Immediately from the definition and by Buchberger's algorithm, we
have the following observation.
\begin{proposition}
If an ideal is almost Buchbergerable, then it is almost Groebner.
\end{proposition}

\begin{proposition}
If $I\subseteq \mathscr{C}[x_1, ..., x_n](S)$ is a Groebner ideal
with Groebner basis $G=\{f_1, ..., f_k\}$, then
$G|_{S-A}=\{f_1|_{S-A}, ..., f_k|_{S-A}\}$ is a Groebner basis of
the ideal $I(S-A)\subseteq \mathscr{C}[x_1, ..., x_n](S-A)$ for any
$A\subseteq S$ of measure zero.
\end{proposition}

\begin{proof}
Let $U=S-A$. By definition, $LC(f_i)$ is nonvanishing on $S$, hence
$LT(f_i)|_U=LT(f_i|_U)$. So for the $S$-polynomial, $S(f_i,
f_j)|_U=S(f_i|_U, f_j|_U)$. By Buchberger's criterion, the remainder
of $S$-polynomial $S(f_i, f_j)$ on division by $G$ is zero, hence
the remainder of $S$-polynomial $S(f_i|_U, f_j|_U)$ on division by
$G|_U$ is zero. Hence by Buchbergerable's criterion, $G|_U$ is a
Groebner basis for the ideal $I(U)$ in $\mathscr{C}[x_1, ...,
x_n](U)$.
\end{proof}

The following simple example shows that we can not make any
extension even for a very perfect function defined in a dense open
set.

\begin{example}
Let $I=<x+s, x-s>\subseteq \mathscr{C}[x](\R)$. Then obviously
$I(\R-\{0\})=\mathscr{C}[x](\R-\{0\})$, and $V(I(\R))=\{(0, 0)\}$.
So $1\in I(\R-\{0\})$ but $1\notin I(\R)$.
\end{example}

\begin{proposition}
\begin{enumerate}
\item A Groebner ideal presheaf is a Groebner ideal sheaf.
\item
An almost Groebner presheaf is an almost Groebner sheaf.
\item
In particular, an almost Buchbergerable presheaf is an almost
Buchbergerable sheaf.
\end{enumerate}
\end{proposition}

\begin{proof}
\begin{enumerate}
\item
Suppose that $I\subseteq \mathscr{C}[x_1, ..., x_n]$ is a
Groebner presheaf induced by a Groebner ideal $I\subseteq
\mathscr{C}[x_1, ..., x_n](S)$. Note that under the measure
topology of $S$, $\mathscr{C}[x_1, ..., x_n]$ is a again a sheaf
on $S$. Let $U$ be an open set of $S$ and $U=\cup_i U_i$ be a
covering of $U$. Suppose that $f_i\in I(U_i)$ such that
$f_i|_{U_i\cap U_j}=f_j|_{U_j\cap U_i}$ for any $i, j$. Then
there is $f \in \mathscr{C}[x_1, ..., x_n](U)$ such that
$f|_{U_i}=f_i$. Let $g_1, ..., g_k$ be a Groebner basis of
$I(U)$. Then $f=\sum^k_{i=1}a_ig_i+r$ for some $a_i, r\in
\mathscr{C}[x_1, ..., x_n](U)$ where $LT(r)$ is not divisible by
any of $LT(g_i)$. Then
$f|_{U_i}=\sum^k_{i=1}a_ig_i|_{U_i}+r|_{U_i}\in I(U_i)$. This
implies that $r|_{U_i}=0$ for all $i$. Thus $r=0$ and hence
$f\in I(S)$. This proves the extension property of $I$. The
uniqueness of $I$ comes directly from the uniqueness of
$\mathscr{C}[x_1, ..., x_n]$.

\item
Let $A\subseteq S$ be a set of measure zero such that $I(S-A)$
is a Groebner ideal of $\mathscr{C}[x_1, ..., x_n](S-A)$. Then
by the result above, $I|_{S-A}$ is a Groebner sheaf and hence
$I$ is an almost Groebner sheaf.

\end{enumerate}
\end{proof}

\begin{proposition}
Let $S$ be a condensed space. Suppose that $I=<f_1, ...,
f_k>\subseteq \mathscr{C}[x_1, ..., x_n](S)$ is a condensed ideal
where $f_1, ..., f_k\in \mathscr{P}(S)$. Then $G=\{f_1, ..., f_k\}$
is Buchbergerable on $S-A$ where $A\subsetneq S$ is a real algebraic
subset. In particular, $I$ is almost Buchbergerable, and $I$ induces
a Buchbergerable almost sheaf.
\end{proposition}

\begin{proof}
Let $G=\{f_1, ..., f_k\}$. The Buchberger's algorithm does not work
if the leading terms vanish somewhere. So the idea of the proof is
to delete the subset of $S$ that leading terms of the
$S$-polynomials vanish and this set is a real algebraic subset of
$S$. Let $f_{12}=\overline{S(f_1, f_2)}^G$ and $C_{12}=$ the zero
locus of the leading coefficient of $f_{12}$. On $S-C_{12}$, let
$G_{12}=G\cup \{f_{12}\}$, and $f_{13}=\overline{S(f_1,
f_3)}^{G_{12}}$. Let $C_{13}$ be the zero locus of the leading
coefficient of $f_{13}$. On $S-(C_{12}\cup C_{13})$, we repeat the
process before. Keep doing this, we get a set $G'=\{f_1, ..., f_k,
f_{ij}, i, j=1, ..., k, i\neq j\}$ and all the leading coefficients
are nonvanishing on $S-\cup_{i\neq j}C_{ij}$. Now repeat the process
for $G=G'$. This process terminates after a finite number of times.
Since $C_{ij}$ are all real algebraic subsets of $S$, we are done.
\end{proof}

\begin{definition}
Suppose that $I, J\subseteq \mathscr{C}[x_1, ..., x_n](S)$ are two
ideals. We say that $I$ is almost contained in $J$, denoted by
$I\subseteq_a J$, if there is a measure zero subset $A\subseteq S$
such that $I(S-A)\subseteq J(S-A)$. We say that $I$ is almost equal
to $J$, denote by $I=_a J$, if $I\subseteq_a J$ and $J\subseteq_a
I$. For $f\in \mathscr{C}[x_1, ..., x_n](S)$, we write $f\in_a I$
and say that $f$ is almost in $I$ if there is a set $A$ of measure
zero such that $f|_{S-A}\in I(S-A)$.
\end{definition}

\begin{definition}
Let $\sum$ be a collection of some ideals of $\mathscr{C}[x_1, ...,
x_n](S)$ where $S$ is a condensed space. If for any chain
$$J_1\subseteq_a J_2 \subseteq_a J_3 \subseteq_a \cdots$$
in $\sum$, there is an $N$ such that $J_{N+k}=_aJ_N$ for all $k\in
\N$, then we say that $\sum$ has the almost ascending chain
property.
\end{definition}

\begin{proposition}
Let $S$ be a condensed space. Let $R=\mathscr{C}[x_1, ..., x_n](S)$,
and $AG(R)$ be the collection of all almost Groebner ideals of $R$.
Then $AG(R)$ has the almost ascending chain property. \label{ASC}
\end{proposition}

\begin{proof}
Let
$$J_1\subseteq_a J_2 \subseteq_a J_3 \subseteq_a \cdots$$ be an ascending
chain in $AG(R)$ and $J=\cup_i J_i$. Suppose that $A_i\subset S$ is
a subset of measure zero such that $J_i(S-A_i)$ is a Groebner ideal
of $\mathscr{C}[x_1, ..., x_n](S-A_i)$. Let
$B=\cup^{\infty}_{i=1}A_i$. Then $B$ has measure zero, and we have
$J(S-B)=\cup^{\infty}_{i=1}J_i(S-B)$.

Let $M=\{LT(f)|f\in J(S-B), LC(f)=1\}$. Since each $J_i(S-B)$ is
Groebner,
$$LT(J(S-B))=<M>\subseteq \C[x_1, ..., x_n](S-B)$$

Consider the monomial ideal generated by $M$ in $\C[x_1, ..., x_n]$.
By Dickson's Lemma \cite[Theorem 2.4.5]{CLS}, it is finitely
generated, hence $LT(J(S-B))$ is finitely generated. Let
$LT(J(S-B))=<a_1, ..., a_t>$ and $g_1, ..., g_t\in J(S-B)$ such that
$LT(g_i)=a_i$. For $h\in J(S-B)$, by the division algorithm, we may
write
$$h=\sum^t_{i=1}b_ig_i+r$$ where $r=0$ or each term of $r$ is not divisible
by any $LT(g_i)=a_i$. If $r\neq 0$, this implies that $LT(r)\notin
LT(J(S-B))$. This contradicts to the fact that $r\in J(S-B)$.
Therefore $r=0$ and hence $J(S-B)$ is generated by $g_1, ...., g_t$.
Suppose that $g_i\in J_{n_i}$ and take $N=max\{n_1, ..., n_t\}$.
Then $J=_a J_N$.
\end{proof}

\begin{corollary}
Let $S$ be a condensed space and $C(R)$ be the collection of all
condensed ideals of $R=\mathscr{C}[x_1, ..., x_n](S)$, $AB(R)$ be
the collection of all almost Buchbergerable ideals of $R$. Then
$C(R), AB(R)$ have the almost ascending chain property.
\end{corollary}

\begin{proof}
This follows from the result above and the fact that $C(R)$ and
$AB(R)$ are subsets of $AG(R)$.
\end{proof}

\begin{proposition}
Suppose that $I, J \subseteq \mathscr{C}[x_1, ..., x_n](S)$ are
condensed ideals. Then $I\subseteq_a J$ if and only if $f\in_a J$
for any $f\in I$.
\end{proposition}

\begin{proof}
Suppose that $I$ is generated by $f_1, ..., f_k$. Then there is a
set $A_i\subseteq S$ of measure zero such that $f_i|_{S-A_i}\in
J(S-A_i)$. Let $A=\cup_{i=1}^kA_i$. Then $f_i|_{S-A}\in J(S-A)$ for
all $i=1, ..., k$. This implies that $I(S-A)\subseteq J(S-A)$.
\end{proof}

\begin{theorem}
If $J$ is almost Groebner, then $\sqrt{J}=_a FRad(J)$.
\end{theorem}

\begin{proof}
It is clear that $\sqrt{J}\subseteq FRad(J)$. Let $f\in FRad(J)(U)$.
Suppose that $J$ is generated by $f_1, ..., f_k\in \mathscr{C}[x_1,
..., x_n](S)$. Let $N=max\{deg \ f_i|i=1, ..., k\}$ be the maximal
degree of all $f_i$. Since $f_s\in \sqrt{J(S)|_s}$, by theorem
\ref{Brownawell}, there is a bound $e$ depends only on $n, k, N$
such that $f^e_s\in J(S)|_s$ for any $s\in S$. Since $J$ is almost
Groebner, there exists real algebraic subset $A\subset S$ such that
$J(S-A)$ is Groebner. Let $g_1, ..., g_l$ be a Groebner basis of
$J(S-A)$. Using the division algorithm we may write
$f^e|_{U-A}=\sum^l_{i=1}a_ig_i+r$ where all $a_i, r\in \CO[x_1, ...,
x_n](U-A)$. For $s\in U-A$, since $\{g_{1, s}, ..., g_{k, s}\}$ is a
Groebner basis for $J(U-A)|_s$, and $f^e_s=\sum^k_{i=1}a_{i, s}g_{i,
s}+r_s\in J(U-A)|_s$, hence $r_s=0$ for $s\in U-A$ which implies
that $f^e|_{U-A}=\sum^l_{i=1}a_ig_i\in J(U-A)$. So $f|_{U-A}\in
\sqrt{J}(U-A)$.
\end{proof}

Combine this theorem and Proposition \ref{Hilbert's
Nullstellensatz}, we generalize the classical Hilbert's
Nullstellensatz to an equality between radical ideals and vanishing
ideals.

\begin{corollary}
If $J$ is an almost Groebner ideal, we have $\sqrt{J}=_a IV(J)$.
\label{Hilbert's Nullstellensatz Groebner ideal}
\end{corollary}

\begin{corollary}(Almost Hilbert Nullstellensatz)
If $I$ is a condensed ideal, then $\sqrt{I} =_a IV(I) =_a FRad(I)$.
\end{corollary}

\begin{proof}
Since $I$ is a condensed ideal, $I$ is almost Buchberger and hence
almost Groebner.
\end{proof}

\begin{proposition}
Let $S$ be a condensed space.
\begin{enumerate}
\item For a condensed ideal $J\subseteq \mathscr{C}[x_1, ..., x_n](S)$, if $g_s\in_a J(S)|_s$ for $s\in S-A$
where $A$ is a set of measure zero, then $g\in_a J$.

\item If $J$ is a radical condensed ideal, i.e, $J=\sqrt{J}$,  then $J|_s=I(\X_s)$ for all $s\in S$ where $\X=V(J)$.

\item In particular, $I(\X)|_s=I(\X_s)$ for all $s\in S$ where $\X$ is
the zero set of some condensed ideal of $\mathscr{C}[x_1, ...,
x_n](S)$.
\end{enumerate}
\end{proposition}

\begin{proof}
\begin{enumerate}
\item Let
$f_1|_{S-A}, ..., f_k|_{S-A}$ be a Groebner basis for $J(S-A)$
where $A$ is a set of measure zero. Write
$g|_{S-A}=\sum^k_{i=1}a_if_i|_{S-A}+r$ where $a_i, r\in
\mathscr{C}[x_1, ..., x_n](S-A)$. Since $g_s=\sum^k_{i=1}a_{i,
s}f_{i, s}+r_s\in J(S-A)|_s$ for $s\in S-A$, this implies that
$r=0$ in $S-A$. Therefore $g\in_a J$.

\item Since $\X_s=V(J|_s)$, so $I(\X_s)=\sqrt{J|_s}=\sqrt{J}|_s=J|_s$.

\item By Hilbert Nullstellensatz, $I(\X)$ is a radical ideal and then the conclusion follows from the result above.
\end{enumerate}
\end{proof}

\subsection{Primary decomposition} In this section,
$R=\mathscr{C}[x_1, ..., x_n](S)$ where $S$ is a condensed space. We
write $C(R)$ for the collection of all condensed ideals of $R$.

\begin{definition}(Elimination ideals)
Let $J\subseteq \mathscr{C}[x_1, ..., x_n](S)$ be an ideal where $S$
is a topological space. The $l$-th elimination ideal $J_l$ is the
ideal defined by
$$J_l:=J\cap \mathscr{C}[x_{l+1}, ..., x_n](S)$$
\end{definition}

The proof of the following result is exactly same as the classical
case (see \cite[Theorem 3.1.2]{CLS}).

\begin{proposition}(The elimination theorem)
Let $J\subseteq \mathscr{C}[x_1, ..., x_n](S)$ be a Groebner ideal
and $G$ be a Groebner basis for $J$ with respect to lex order where
$x_1>x_2> \cdots >x_n$. Then for any $\ell$ where $0\leq \ell \leq
n$, the set
$$G_{\ell}:=G\cap \mathscr{C}[x_{\ell}, ..., x_n](S)$$ is a Groebner basis of
the $\ell$-th elimination ideal $J_{\ell}$.
\end{proposition}

For the following proof, see \cite[Theorem 4.3.11]{CLS}.
\begin{proposition}
Let $I, J$ be ideals in $\mathscr{C}[x_1, ..., x_n](S)$. Then
$$I\cap J=(tI+(1-t)J)\cap \mathscr{C}[x_1, ..., x_n](S)$$
\end{proposition}

\begin{corollary}
If $I, J$ are two condensed ideals in $\mathscr{C}[x_1, ...,
x_n](S)$, then $I\cap J$ is an almost condensed ideal.
\label{intersecting ideals}
\end{corollary}

\begin{proof}
Let $\{f_1, ..., f_k\}, \{g_1, ..., g_l\}$ be condensed bases for
$I$ and $J$ respectively which restricted to $S-A$ are Groebner
bases of $I(S-A)$ and $J(S-A)$ respectively for some real algebraic
subset $A\subseteq S$. Consider the condensed ideal
$$L=<tf_1, ..., tf_k, (1-t)g_1, ...,(1-t)g_l>$$ in $\mathscr{C}[x_1,
..., x_n, t](S)$. By the continuous Buchberger's algorithm Lemma
\ref{algorithms}, compute a Groebner basis $G$ for $L(S-A)$ with
respect to lex order in which $t$ is greater than all $x_i$. By
Proposition above, the elements of $G$ which do not contain the
variable $t$ will form a basis of $I(S-A)\cap J(S-A)=(I\cap
J)(S-A)$.
\end{proof}

\begin{definition}(Almost irreducible semi-topological algebraic sets)
For $J\in C(R)$ a condensed ideal, $V(J)$ is called a
semi-topological algebraic set. If $V(J)\neq_a V(J_1)\cup V(J_2)$
for any $J_1, J_2\in C(R)$ with $V(J_1), V(J_2)\neq_a V(J)$, then we
say $V(J)$ is almost irreducible. Otherwise, $V(J)$ is said to be
almost reducible.
\end{definition}

The following result is a simple consequence of the almost ascending
chain property.

\begin{proposition}
If $J\in C(R)$, then $V(J)$ can be written as a finite union of
almost irreducible affine semi-topological algebraic sets.
\end{proposition}

\begin{definition}(Almost prime ideals)
Let $J\in C(R)$. The ideal $J$ is said to be almost prime if for
condensed $f, g\in \mathscr{R}[x_1, ..., x_n](S)$, $fg\in_a J$, then
either $f\in_a J$ or $g\in_a J$.
\end{definition}

If we take $S$ to be a point, then $R=\mathscr{R}(S)[x_1, ...,
x_n]=\C[x_1, ..., x_n]$ and the above definition reduces to the
usual definition of prime ideals.

\begin{proposition}
Let $J\in C(R)$ and $V=V(J)$.
\begin{enumerate}
\item If $J$ is an almost radical ideal, i.e., $J=_a\sqrt{J}$, then $J$ is almost prime if and only if $V$ is
almost irreducible.
\item For $J\in C(R)$, write $V(J)=_a V_1\cup \cdots \cup V_k$ where $V_i$ are
almost irreducible components of $V(J)$. Then
$\sqrt{J}=_aP_1\cap \cdots \cap P_k$ where $P_i=I(V_i)$ for
$i=1, ..., k$.
\end{enumerate}\label{irreducible components}
\end{proposition}

\begin{proof}
\begin{enumerate}
\item
Suppose that $J$ is almost prime. If $V$ is almost reducible,
let $V=_aV(J_1)\cup V(J_2)$ where $J_1, J_2\in C(R)$ and
$V(J_1)\subsetneq_a V, V(J_2)\subsetneq_a V$. Then there exists
$f\in_a J_1, g\in_a J_2$, $f, g$ condensed, which are not in
$J(S-A)$ for any measure zero set $A$. Because of $(fg)(V)=_a
0$, $fg\in_a IV(J)=_a \sqrt{J}=_a J$. So $f\in_a J$ or $g\in_a
J$ which is a contradiction. Hence $V$ is almost irreducible.

Now assume that $V$ is irreducible. Let $f, g\in
\mathscr{R}[x_1, ..., x_n](S)$ such that $fg\in_a J$. Since
$V=_aV\cap V(fg)=_aV\cap V(f)\cup V\cap V(g)$, by Corollary
\ref{intersecting ideals} and the irreducibility of $V$, we must
have $V\cap V(f)=_aV$ or $V\cap V(g)=_aV$. By Hilbert
Nullstellensatz, this implies that $f$ or $g$ are almost in
$\sqrt{J}=J$. Hence $J$ is almost prime.

\item Let $J'=_a P_1\cap \cdots \cap P_k$. Then for $f\in_a J'$,
$f(V_i)=_a 0$ for all $i$ which implies that $f\in_a \sqrt{J}$.
If $g\in_a \sqrt{J}$, then $g(V_i)=_a 0$ for all $i$ and hence
$g\in_a I(V_i)=P_i$ for all $i$ which implies that $g\in_a J'$.
Therefore $J'=_a\sqrt{J}$.
\end{enumerate}
\end{proof}

The following result is of fundamentally important in our theory.

\begin{theorem}
If $J\subseteq \mathscr{C}[x_0, ...., x_n](S)$ is a condensed ideal,
then $\sqrt{J}$ is almost condensed in $\mathscr{C}[x_0, ...,
x_n](S)$.
\end{theorem}

\begin{proof}
Let $V(J)=_a V_1\cup \cdots \cup V_k$ be a decomposition into almost
irreducible components. By Proposition \ref{irreducible components},
$\sqrt{J}=_a P_1\cap \cdots \cap P_k$ where $V_i=V(P_i)$ and $P_i$
is an almost prime. Hence by Corollary \ref{intersecting ideals}, we
know that $\sqrt{J}$ is almost condensed in $\mathscr{C}[x_0, ...,
x_n](S)$.
\end{proof}

\begin{corollary}
If $J\subseteq \mathscr{C}[x_1, ...., x_n](S)$ is an almost Groebner
ideal, then $\sqrt{J}$ is almost Groebner.
\end{corollary}

\begin{proof}
Since $J$ is almost Groebner, it is almost condensed, hence
$\sqrt{J}$ is almost condensed and hence almost Groebner.
\end{proof}

\begin{definition}
Let $I, J\subseteq \mathscr{C}[x_1, ..., x_n](S)$ be two condensed
ideals where $S$ is a condensed space. The ideal quotient of $I$ by
$J$ is defined by
$$I:J=\{f\in \mathscr{C}[x_1, ..., x_n](S)|fg\in I, \forall g\in J\}$$
\end{definition}

We refer the reader to \cite[Theorem 4.4.11]{CLS} for the proof of
the following result.

\begin{proposition}
Let $J\in \mathscr{C}[x_1, ..., x_n](S)$ be an ideal where $S$ is a
condensed space. If $\{f_1, ..., f_k\}$ is a condensed basis of the
ideal $J\cap <g>$ where $f_1, ..., f_k, g\in \mathscr{R}(S)[x_0,
..., x_n]$, then $\{f_1/g, ..., f_k/g\}$ is a condensed basis of
$I:<g>$. \label{quotient single}
\end{proposition}

\begin{proposition}
Let $I, J\subseteq \mathscr{C}[x_0, ..., x_n](S)$ be condensed
ideals. Then $I:J$ is also an almost condensed ideal.
\label{quotient ideal}
\end{proposition}

\begin{proof}
Let $J=<g_1, ..., g_l> $ where $g_1, ..., g_l\in \mathscr{R}(S)[x_1,
..., x_n]$. By Proposition \ref{intersecting ideals}, $I\cap <g_i>$
is almost condensed for all $i$, then by Proposition \ref{quotient
single}, $I:<g_i>$ is almost condensed. Since $I:<g_1,
g_2>=(I:<g_1>)\cap (I:<g_2>)$ which is almost condensed, by
induction, $I:J$ is almost condensed.
\end{proof}

\begin{definition}(Almost primary ideals)
Let $J\subseteq \mathscr{C}[x_0, ..., x_n](S)$ be a condensed ideal
where $S$ is a condensed space. Suppose that for any $f,g\in
\mathscr{R}(S)[x_0, ..., x_n]$, $fg\in_a J$, we have $f\in_a J$ or
$g^k\in_a J$ for some $k\in \N$. Then we say that $J$ is an almost
primary ideal. \label{primary}
\end{definition}

One of the cornerstone in commutative ring theory is that every
ideal in a commutative Noetherian ring has a primary decomposition.
We show that every condensed ideal has an almost primary
decomposition.

\begin{theorem}(Almost primary decomposition)
Every condensed ideal $J\subseteq \mathscr{C}[x_1, ..., x_n](S)$ can
be written as a finite intersection of almost primary ideals where
$S$ is a condensed space.
\end{theorem}

\begin{proof}
We say that a Groebner ideal $J$ is almost irreducible if $J=_a
J_1\cap J_2$ for some condensed ideals $J_1, J_2$, then either $J=_a
J_1$ or $J=_a J_2$. By using the almost ascending chain property of
$C(R)$, it is not difficult to show that every condensed ideal is an
intersection of finitely many almost irreducible condensed ideals.
The second step is to show that an almost irreducible condensed
ideal is almost primary. Suppose that $J$ is an almost irreducible
condensed ideal, $f, g\in \mathscr{R}(S)[x_1, ..., x_n]$ and $fg
\in_a J$. We assume that $f\notin_a J$. There is an ascending chain:
$$J:g\subseteq_a J:g^2 \subseteq_a \cdots$$
By the almost ascending chain property of $C(R)$, there is $N\in \N$
such that $J:g^N=_a J:g^{N+1}$. For $a+h_1g^N= b+h_2f$ where $a,
b\in_a J, h_1, h_2\in_a R$, we have $ga+h_1g^{N+1}=bg+h_2fg\in_a J$.
Hence $h_1g^{N+1}\in_a J$ which implies that $h_1\in_a
J:g^{N+1}=J:g^N$, so we have $h_1g^N\in_a J$. This tells us that
$(J+<g^N>)\cap (J+<f>)=_a J$. Since $J$ is almost irreducible and
$f\notin_a J$, we have $J=_a J+<g^N>$. Hence $g^N\in_a J$.
\end{proof}

\section{Semi-topological Zariski topology}
In this section, we consider a condensed space $S$ with its measure
topology, i.e., closed subsets of $S$ are countably intersection of
countably union of real algebraic subsets of $S$. Such sets are said
to have measure zero. We remind the reader that $\mathscr{C}(S)$ is
the ring of all complex-valued continuous functions from $S$ to $\C$
under the Euclidean topology of $S$.

\begin{definition}
Let $\mathscr{F}$ be a presheaf on a sigma space $S$. If there is a
set $A\subseteq S$ of measure zero such that $\mathscr{F}|_{S-A}$ is
a sheaf, then $\mathscr{F}$ is called a strong almost sheaf.
\end{definition}

\begin{remark}
It is clear that on a condensed space, a strong almost sheaf is an
almost sheaf.
\end{remark}

\begin{proposition}
If $\X\subseteq \A^n_S$ is an affine semi-topological algebraic set
over $S$, then $K[\X]$ is a strong almost sheaf on $S$.
\end{proposition}

\begin{proof}
Since $I(\X)$ is an almost Groebner sheaf, there is a set
$A\subseteq S$ of measure zero such that $I(\X)|_{S-A}$ is a
Groebner ideal sheaf. Let $g_1, ..., g_k$ be a Groebner basis of
$I(\X)(S-A)$. To prove the extension, let $U$ be an open set,
$U-A=\cup_i U_i$ and $\sigma_i\in K[\X](U_i)$ for each $i$ such that
$\sigma_i|_{U_i\cap U_j}=\sigma_j|_{U_j\cap U_i}$. Let
$\sigma_i=f_i+I(\X)(U_i)$. Write $f_i=\sum_ja^i_jg_j|_{U_i}+r_i$
where $a^i_j, r_i\in \mathscr{C}[x_1, ..., x_n](U_i)$ and each term
of $r_i$ is not divisible by any of $LT(g_i|_{U_i})$. Since
$f_i|_{U_i\cap U_j}-f_j|_{U_j\cap U_i}\in I(\X)(U_i\cap U_j)$, we
have $r_i|_{U_i\cap U_j}=r_j|_{U_j\cap U_i}$. Hence there is $r\in
\mathscr{C}[x_1, ..., x_n](U-A)$ such that $r|_{U_i}=r_i$, and
$(r+I(\X)(U-A))|_{U_i}=\sigma_i$ for each $i$. To prove the
uniqueness, let $\sigma\in K[\X](U)$ and suppose there is an open
covering $U=\cup_i U_i$ such that $\sigma|_{U_i}=0$ for all $i$. Let
$\sigma|_{U_i}=f_i+I(\X)(U_i)$. Then $f_i\in I(\X)(U_i)$ and
$f_i|_{U_i\cap U_j}=f_j|_{U_j\cap U_i}$ for any $i, j$. Since
$I(\X)$ is an almost sheaf, there is $f\in I(\X)(U-A)$ such that
$f|_{U_i-A}=f_i|_{U_i-A}$ for all $i$. Thus
$\sigma|_{U-A}=f|_{U-A}+I(\X)(U-A)=I(\X)(U-A)$.
\end{proof}

\begin{definition}(Associated sheaves)
Suppose that $S$ is a sigma space and $\mathscr{F}$ is an almost
sheaf on $S$. For $\alpha\in_a \mathscr{F}(U)$, write
$$[\alpha]=\{\beta|\beta \in_a \mathscr{F}(V)\mbox{ for some $V$ and }
\beta=_a \alpha\}$$

Define
$$\mathscr{AF}(U):=\{[\alpha]|\alpha\in_a \mathscr{F}(U)\}$$
\end{definition}

\begin{proposition}
Suppose that $\mathscr{F}$ is an almost sheaf of rings on a sigma
space $S$. For $[\alpha], [\beta]\in \mathscr{AF}(U)$, $\alpha\in
\mathscr{F}(W), \beta\in \mathscr{F}(V)$, define
$$[\alpha]+[\beta]:=[\alpha|_{W\cap V}+\beta|_{W\cap V}]$$
$$[\alpha][\beta]:=[\alpha|_{W\cap V}\beta|_{W\cap V}]$$
And define restriction maps $[\alpha]|_{U'}:=[\alpha|_{U'\cap W}]$
where $U'\subseteq U$. Then $\mathscr{AF}$ is a sheaf of rings on
$S$.
\end{proposition}

The sheaf $\mathscr{AF}$ is called the sheaf associated to
$\mathscr{F}$.

\begin{definition}
Suppose that $\mathscr{F}, \mathscr{G}$ are two almost sheaves on a
sigma space $S$. Let $\varphi:\mathscr{F} \rightarrow \mathscr{G}$
be a morphism, i.e., a morphism of presheaves.
\begin{enumerate}
\item We say that $\varphi$ is almost injective if $\sigma
\in \mathscr{F}(U)$ and $\varphi_U(\sigma)=_a 0$, then $\sigma
=_a 0$.

\item We say that $\varphi$ is almost surjective if for any
$\beta\in \mathscr{G}(U)$, there is an open covering $U=\cup_i
U_i$ and $\alpha_i\in \mathscr{F}(U_i)$ such that
$\varphi_{U_i}(\alpha_i)=_a \beta|_{U_i}$.
\end{enumerate}

If $\varphi$ is almost injective and almost surjective, then we say
that $\varphi$ is an almost isomorphism and say that $\mathscr{F},
\mathscr{G}$ are almost isomorphic, denoted by $\mathscr{F}\cong_a
\mathscr{G}$. If $\varphi: \mathscr{F} \rightarrow \mathscr{G}$ is a
morphism between almost sheaves, then it induces a morphism
$\mathscr{A}\varphi:\mathscr{AF} \rightarrow \mathscr{AG}$ between
the associated sheaves defined by
$$(\mathscr{A}\varphi)_U[\alpha]:=[\varphi_U(\alpha)]$$
\end{definition}

It is not difficult to see that $\varphi$ is almost injective,
surjective, isomorphic if and only if $\mathscr{A}\varphi$ is
injective, surjective, isomorphic respectively. We note that if
$\varphi$ is almost isomorphic, in the level of almost sheaves,
there may no exist an inverse of $\varphi$. But the inverse of
$\mathscr{A}\varphi$ exists.

\begin{remark}
From the definition, we see that $\mathscr{AF}(U)=\mathscr{AF}(U-A)$
where $A$ is a set of measure zero.
\end{remark}

\begin{definition}
Let $S$ be a condensed space. If $I\subseteq \mathscr{C}[x_1, ...,
x_n](S)$ is a condensed ideal, the set $V(I)\subseteq \A^n_{S}$ is
called an affine semi-topological algebraic set. The almost sheaf
$K[\X]:=\mathscr{C}[x_1, ..., x_n]/I(\X)$ on $S$ is called the
coordinate almost sheaf of $\X$. Call $\mathscr{A}K[\X]$ the
coordinate sheaf of $\X$. An element $\sigma$ of $K[\X](U)$ is said
to be condensed if there is a condensed element $g\in
\mathscr{C}[x_1, ..., x_n](U)$ such that $\sigma=g+I(\X)(U)$. And an
element $[\sigma]\in \mathscr{A}K[\X](U)$ is said to be condensed if
it can be represented by some condensed element in $K[\X](U)$.
\end{definition}

\begin{definition}(Topology for affine semi-topological algebraic sets)
Let $\X\subseteq \A^n_S$ be an affine semi-topological algebraic set
over $S$. For an open set $U\subseteq S$, $f \in K[\X](U)$
condensed, the principal open set defined by $f$ is the set
$$D(f):=\X(U)-Z(f)$$ where $Z(f):=\{(s, x)\in \X(U)|f_s(x)=0\}$
is the zero set of $f$ over $U$. Then the collection of all sets of
the form $D(f)$ forms a basis of a topology on $\X$. We call the
topology generated by this basis the semi-topological Zariski
topology of $\X$.
\end{definition}

\subsection{Sheaves of semi-topological modules}

\begin{definition}(Affine localization)
Given an affine semi-topological algebraic set $\X$ over $S$ and a
presheaf $M$ of $K[\X]$-modules on $S$. Let $U\subseteq S$ be an
open set. For $p\in \pi^{-1}(U)$, let
$$m_{[p]}(U):=\{f\in K[\X](U)|f(p)\neq 0, f \mbox{ condensed}\}$$ then $m_p(U)$ is a
multiplicative system. Define
$$M_{[p]}(U):=(m_{[p]}(U))^{-1}M(U)$$
the localization of $M(U)$ with respect to $m_{[p]}(U)$.

Then $M_{[p]}$ is a presheaf on $S$ which is called the localization
of $M$ at $p$.
\end{definition}

We write $(M_{[p]})_{\pi(p)}$ for the stalk of $M_{[p]}$ at
$\pi(p)\in S$.

\begin{definition}(Bivariant sheaves)
Let $\X$ be a semi-topological scheme over $S$. A bivariant sheaf
$\mathscr{F}$ on $\X$ is an assignment such that
\begin{enumerate}
\item
$\mathscr{F}(U)$ is a sheaf for all $U$ open in $\X$. Let
$r^U_{VW}:\mathscr{F}(U)(V) \rightarrow \mathscr{F}(U)(W)$ be
the restriction map where $W\subseteq V$ are open sets in
$\pi(U)$.

\item
For any $U'\subseteq U$ open sets of $\X$, there exists
restriction map $r^{UU'}_V:\mathscr{F}(U)(V) \rightarrow
\mathscr{F}(U')(V)$ such that the following diagram commutes:
$$\xymatrix{ \mathscr{F}(U)(V) \ar[r]^{r^U_{VW}} \ar[d]_{r^{UU'}_V}
& \mathscr{F}(U)(W) \ar[d]^{r^{UU'}_W}\\
\mathscr{F}(U')(V) \ar[r]^{r^{U'}_{VW}} & \mathscr{F}(U')(W)}$$
where $W\subseteq V$ are open sets in $\pi(U')$.
\end{enumerate}
\end{definition}

\begin{definition}
Suppose that $M$ is a presheaf of $K[\X]$-modules on $S$. Let
$V\subseteq \X$, $U\subseteq S$ be open sets. Let
$\widetilde{M}(V)(U)$ be the collection of all functions
$\sigma:V\cap \pi^{-1}(U) \rightarrow \coprod_{p\in V\cap
\pi^{-1}(U)}(M_{[p]})_{\pi(p)}$ such that each of them satisfies the
following conditions
\begin{enumerate}
\item $\sigma(p)\in (M_{[p]})_{\pi(p)}$ for each $p\in V\cap \pi^{-1}(U)$;

\item there exists open covering $V\cap \pi^{-1}(U)=\cup_i W_i$, $f_i\in K[\X](\pi(W_i)) \mbox{ condensed }, 0\notin f_i(W_i), a_i\in
M(\pi(W_i))$ such that $\sigma(q)=\frac{a_i}{f_i}\in
M_{[q]}(\pi(W_i))$ if $q\in W_i$.
\end{enumerate}
We call $\widetilde{M}$ the presheaf of sheaves on $\X$ associated
to $M$. If $M$ is an almost sheaf of $K[\X]$-modules, then we call
$\widetilde{M}$ the almost sheaf of sheaves on $\X$ associated to
$M$. If $M$ is a sheaf, then $\widetilde{M}$ is a bivariant sheaf.
\end{definition}

\begin{definition}(Structure sheaves)
Let $\X$ be an affine semi-topological algebraic set over $S$.
Define $\mathcal{O}_{\X}:=\widetilde{\mathscr{A}K[\X]}$ which is
called the bivariant structure sheaf of $\X$. If $M$ is an almost
sheaf of $K[\X]$-modules, then $\widetilde{\mathscr{A}M}$ is called
the bivariant sheaf of $\mathcal{O}_{\X}$-modules associated to $M$.
\end{definition}

\begin{remark}
All presheaves on a condensed space are almost sheaves. We have the
following two functors:
$$\mathscr{A}:Cat(Ash K[\X]-mod) \rightarrow
Cat(\mathscr{A}K[\X]-mod)$$
$$\widetilde{ }:Cat(\mathscr{A}K[\X]-mod) \rightarrow Cat(BSh
\mathcal{O}_{\X}-mod)$$ where $Cat(Ash K[\X]-mod)$ is the category
of almost sheaves of sheaves of $K[\X]$-modules,
$Cat(\mathscr{A}K[\X]-mod)$ is the category of sheaves of
$\mathscr{A}K[\X]$-modules, and $Cat(BSh \mathcal{O}_{\X}-mod)$ is
the category of bivariant sheaves of $\mathcal{O}_{\X}$-modules.
\end{remark}

\begin{definition}(Almost Noetherian space)
Let $S$ be a condensed space and $\pi:\X \rightarrow S$ a space over
$S$. The space $\X$ is said to be an almost Noetherian space if for
any chain of open sets $U_1\subseteq U_2 \subseteq \cdots $, there
is a set $A\subseteq S$ of measure zero and a number $N$ such that
$U_N-\pi^{-1}(A)=U_{N+k}-\pi^{-1}(A)$ for all $k>0$.
\end{definition}

\begin{proposition}
Suppose that $\X\subseteq \A^n_S$ is an affine semi-topological
algebraic set over $S$, then $\X$ is almost Noetherian.
\end{proposition}

\begin{proof}
Let $P:\mathscr{C}[x_1, ..., x_n] \rightarrow K[\X]$ be the natural
quotient map. First we note that $K[\X]$ has the almost ascending
chain property, i.e., for any open set $U\subseteq S$, if
$$I_1\subseteq I_2 \subseteq \cdots $$ is a chain of condensed ideals of $K[\X](U)$, since $I(\X)(U)$ is almost
Groebner,
$$P^{-1}(I_1) \subseteq P^{-1}(I_2) \subseteq P^{-1}(I_3) \subseteq
\cdots $$ is a chain of almost Groebner ideals in $\mathscr{C}[x_1,
..., x_n](U)$, hence by the almost ascending chain property, the
chain almost stabilizes at some $N$ and so is the chain in
$K[\X](U)$.

If
$$D(f_1) \subseteq \cup^2_{i=1} D(f_i) \subseteq \cup^3_{i=1}D(f_i)
\subseteq \cdots$$ is a chain of open sets in $\X$ where $f_i\in K[\X](S-A_i)$ where $A_i$ is of measure zero, we get a chain
$$<f_1|_{S-A}>\subseteq <\{f_i|_{S-A}\}^2_{i=1}>\subseteq
<\{f_i|_{S-A}\}^3_{i=1}>\subseteq \cdots $$ in $K[\X](S-A)$ where
$A=\cup_i A_i$. Then by the almost ascending chain property of
$K[\X]$, the sequence almost stabilizes at some $N$.

Since the semi-topological Zariski topology of $\X$ is generated by
sets of the form $D(f_i)$, an open set $U\subseteq \X$ can be
written as
$$U=\cup^{\infty}_{i=1}D(f_i)$$ for countably many $f_i$. Then
$$D(f_1) \subseteq \cup^2_{i=1}D(f_i) \subseteq \cup^3_{i=1}D(f_i)
\subseteq \cdots $$ and hence $U=_a \cup^N_{i=1}D(f_i)$ for some $N$.

Hence without loss of generality, it is enough to consider chains of
open sets of $\X$ of the form
$$D(f_1) \subseteq \cup^2_{i=1}D(f_i) \subseteq \cup^3_{i=1}D(f_i)
\subseteq \cdots $$ but then by previous argument, we know that this sequence almost stabilizes at some $N$. Hence $\X$ is almost Noetherian.
\end{proof}

\begin{proposition}
Let $\X$ be an affine semi-topological algebraic set over a
condensed space $S$ and $M$ an almost sheaf of $K[\X]$-modules. Then
\begin{enumerate}
\item For $p\in \X$, the stalk $(\widetilde{M})_{p}$ is isomorphic to the stalk $(M_{[p]})_{\pi(p)}$ as a
$(K[\X]_{[p]})_{\pi(p)}$-module.

\item If $f\in K[\X]$ is condensed, then $M_f\cong_a \widetilde{M}(D(f))$ as almost sheaves of $K[\X]_f$-modules.

\item In particular, $M\cong_a \widetilde{M}(\X)$ as almost sheaves of $K[\X]$-modules.
\end{enumerate}\label{module principal open set}
\end{proposition}

\begin{proof}
\begin{enumerate}
\item For $\sigma\in (\widetilde{M})_{p}$, there exists a neighborhood
$U\subseteq \X$ of $p$ such that $\sigma\in \widetilde{M}(U)$.
Define $\varphi(\sigma)=\sigma(p)$, then we get a well defined
$(K[\X]_{[p]})_{\pi(p)}$-homomorphism
$\varphi:(\widetilde{M})_{p} \rightarrow (M_{[p]})_{\pi(p)}$. To
prove that $\varphi$ is surjective, consider $\alpha\in
(M_{[p]})_{\pi(p)}$. Then $\alpha=\frac{a}{f}$ in a neighborhood
$W\subseteq S$ of $\pi(p)$ where $a\in M(W), f\in K[\X](W),
f(p)\neq 0$, $f$ condensed on $W$. Hence $D(f)\subseteq \X$ is a
neighborhood of $p$ and $\frac{a}{f}\in \widetilde{M}(D(f))$
which implies $\frac{a}{f}\in (\widetilde{M})_{p}$ and hence
$\varphi(\frac{a}{f})=\alpha$. To show that $\varphi$ is
injective, we assume that $\sigma, \eta\in (\widetilde{M})_{p}$
such that $\sigma(p)=\eta(p)$. Then we may take a neighborhood
$U\subseteq \X$ of $p$ such that $\sigma=\frac{a}{f},
\eta=\frac{b}{g}$ in $\widetilde{M}(U)$ where $a, b\in M(U), f,
g\in m_{[p]}(U)$. Then there is $h\in m_{[p]}(U)$ such that
$h(ag-bf)=0$ which implies that $\frac{a}{f}=\frac{b}{g}$ in the
neighborhood $W'=D(f)\cap D(g)\cap D(h)$ of $p$. Then
$\sigma=\eta$ in $\widetilde{M}(W)$ and therefore $\sigma=\eta$
in $(\widetilde{M})_{p}$.

\item
Let $U$ be an open subset of $S$ and $f\in K[\X](U)$ condensed
in $U$. If $\eta=\frac{a}{f^k}\in M(U)_f$ for some $a\in M(U)$,
$\eta$ induces an element $\widetilde{\eta}\in
\widetilde{M}(D(f))(U)$ that assigns each $p\in D(f)\cap
\pi^{-1}(U)$, the element $\frac{a}{f^k}\in (M_{[p]})_{\pi(p)}$.
Define $\psi: M(U)_f \rightarrow \widetilde{M}(D(f))(U)$ by
sending $\eta$ to $\widetilde{\eta}$.

Suppose that $\psi(\frac{a}{f^k})=\psi(\frac{b}{f^l})$. Then
$\frac{a}{f^k}=\frac{b}{f^l}\in (M_{[p]})_{\pi(p)}$ for $p\in
D(f)\cap \pi^{-1}(U)$. So there is an open neighborhood
$V\subseteq U$ of $\pi(p)$ such that
$\frac{a}{f^k}=\frac{b}{f^l}$ in $M_{[p]}(V)$. And hence there
is $h\in m_{[p]}(V)$ such that $h(f^la-f^kb)=0$ in $M(V)$. Let
$$\mathscr{A}=\{g\in K[\X](V)| g \mbox{ condensed }, g(f^la-f^kb)=0\}$$ and
$I$ be the ideal of $K[\X](V)$ generated by $\mathscr{A}$. Since
for any $q\in D(f|_V)$, there is $g\in \mathscr{A}$ such that
$g(q)\neq 0$. Therefore $V(I)\cap D(f|_V)=\emptyset$. We have
$f(V(I))=0$. Since $I$ is a condensed ideal, by Hilbert
Nullstellensatz, $f|_V^N\in_a \mathscr{A}$. Hence
$f^N(f^la-f^kb)=_a0$ in $M(V)$ which implies that
$\frac{a}{f^k}=_a\frac{b}{f^l}$ in $M(U)_f$. This proves that
$\psi$ is almost injective.

Let $\sigma \in \widetilde{M}(D(f))(U)$. Then there exists open
cover $\{V_i\}$ of $D(f)$ such that on $V_i$,
$\sigma|_{V_i}=\frac{a_i}{g_i}$ where $g_i\in K[\X](\pi(V_i))$
is not vanishing in $V_i$ and condensed in $\pi(V_i)$. Hence
$V_i\subseteq D(g_i)$. Since the semi-topological Zariski
topology of $\X$ is generated by principal open sets, we may
assume $V_i=D(h_i)$ for some $h_i\in K[\X](U)$ where $h_i$ is
condensed in $\pi(U)$. Then $D(h_i)\subseteq D(g_i)$, we have
$h_i\in \sqrt{<g_i>}$ for each $i$. So $h^m_i=cg_i$ for some
$c\in \mathscr{C}[x_1, ..., x_n](U)$ and
$\frac{a_i}{g_i}=\frac{ca_i}{h^m_i}$. Replacing $h_i$ by $h^m_i$
and $a_i$ by $ca_i$, we may assume that $D(f)$ is covered by the
open subsets $D(h_i)$, and that $\sigma$ is represented by
$\frac{a_i}{h_i}$ in $D(h_i)$.

Since $\X$ is almost Noetherian, we may assume that $D(f)$ is
almost covered by $D(h_1), ..., D(h_t)$ where each $h_i\in
K[\X](W)$.

On $D(h_i)\cap D(h_j)=D(h_ih_j)$, $\frac{a_i}{h_i}$ and
$\frac{a_j}{h_j}$ both represent $\sigma$. Applying the
injectivity proved above, we have
$\frac{a_i}{h_i}=\frac{a_j}{h_j}$ in
$M(\pi(D(h_ih_j)))_{h_ih_j}$. Therefore, there is $m\in \N$ such
that $(h_ih_j)^m(h_ja_i-h_ia_j)=0$ which implies
$h^{m+1}_j(h^m_ia_i)-h^{m+1}_i(h^m_ja_j)=0$. If we replace $a_i$
by $h^m_ia_i$ and $h_i$ by $h^{m+1}_i$, then on $D(h_i)$,
$\sigma$ is again represented by $\frac{a_i}{h_i}$ and we have
$h_ja_i=h_ia_j$ on $D(h_i)\cap D(h_j)$.

Let $J=<h_1, ..., h_t>\subseteq K[\X](W)$ be the condensed ideal
generated by $h_1, ..., h_t$. Since $D(f|_W)\subseteq
\cup^t_{i=1}D(h_i)$, so $f|_W\in \sqrt{J}$ this implies that
$f^m|_W=\sum^t_{i=1}b_ih_i$ for some $b_i\in K[\X](W)$. Let
$a=\sum^t_{i=1}b_ia_i\in K[\X](W)$. Then on $D(h_j)$,
$h_ja=\sum^t_{i=1}b_ih_ja_i=\sum^t_{i=1}b_ih_ia_j=f^m|_Wa_j$. So
$$\frac{a}{f^m|_W}=\frac{a_j}{h_j}$$ on $D(h_j)$. Since
$\frac{a}{f^m|_W}\in M(W)_f$, this shows that $\psi$ is almost
surjective.

\item For any $U\subseteq S$ open, by result above, we have
$\widetilde{M}(\X)(U)=\widetilde{M}(\pi^{-1}(U))(U)\cong M(U)$
and hence $M$ is almost isomorphic to $\widetilde{M}(\X)$ as
$K[\X]$-modules.
\end{enumerate}
\end{proof}

If $M$ is an almost sheaf, and $\alpha\in M(U)$, we write $[\alpha]$
for the class of $\alpha$ in $\mathscr{A}M(U)$.

\begin{lemma}
Suppose that $M$ is an almost sheaf of $K[\X]$-modules over a
condensed space $S$ where $\X$ is a semi-topological algebraic set.
Then $\mathscr{A}(\widetilde{M})\cong \widetilde{\mathscr{A}M}$.
\end{lemma}

\begin{proof}
For any $[\sigma] \in \mathscr{A}(\widetilde{M})(U)$, there is an
open cover $(U)=\cup_i V_i$, $U-A=\cup_i(V_i-A)$, and $a_i\in
M(\pi(V_i)), f_i\in K[\X](\pi(V_i))$, $0\notin f_i(V_i)$ such that
$\sigma(p)=\frac{a_i}{f_i}\in (M_{[p]})_{\pi(p)}$ for $p\in V_i$. We
define $\sigma'(p)=\frac{[a_i]}{f_i}\in
((\mathscr{A}M)_{[p]})_{\pi(p)}$. Then $\sigma' \in
\widetilde{\mathscr{A}M}(U)$. Not difficult to check the assignment
$\sigma\mapsto \sigma'$ defines an isomorphism
$$\mathscr{A}(\widetilde{M})\rightarrow \widetilde{\mathscr{A}M}$$
\end{proof}

\begin{corollary}
Suppose $\X$ is a semi-topological algebraic set over $S$. Then
$\mathscr{A}M \cong \widetilde{\mathscr{A}M}(\X)$.
\end{corollary}

\begin{proof}
Since $M\cong_a \widetilde{M}(\X)$, $\mathscr{A}{M}\cong
\mathscr{A}\widetilde{M}(\X)\cong \widetilde{\mathscr{A}M}(\X)$.
\end{proof}

\subsection{Bivariant coherent sheaves}
\begin{definition}
Let $\X$ be an affine semi-topological algebraic set over $S$, and
$\mathscr{F}$ a sheaf of $\mathcal{O}_{\X}$-modules. We say that
$\mathscr{F}$ is quasi-coherent if $\X$ can be covered by affine
open subsets $U_i$ such that for each $i$,
$\widetilde{\mathscr{A}M}_i\cong \mathscr{F}|_{U_i}$ where $M_i$ is
an almost sheaf of $K[U_i]$-module. We say that $\mathscr{F}$ is
coherent if furthermore each $M_i$ can be taken to be an almost
sheaf of finitely generated $K[U_i]$-module.
\end{definition}

We denote by $\mathscr{AS}(\X)$ the category of associated sheaves
of almost sheaves of finitely generated $K[\X]$-modules and
$\mathscr{CS}(\X)$ the category of bivariant coherent sheaves of
$\mathcal{O}_{\X}$-modules.

\begin{proposition}
Let $\X, \Y$ be affine semi-topological algebraic sets over a
condensed space $S$.

\begin{enumerate}
\item The map $\mathscr{A}M \rightarrow \widetilde{\mathscr{A}M}$ is an exact fully
faithful functor from $\mathscr{AS}(\X)$ to $\mathscr{CS}(\X)$.

\item If $M, N$ are $K[\X]$-modules, then $(M\otimes_{K[\X]} N)\widetilde \ \cong
\widetilde{M}\otimes_{\mathcal{O}_{\X}}\widetilde{N}$.

\item If $\{M_i\}$ is any family of
$\mathcal{O}_{\X}$-modules, then $(\bigoplus M_i)\widetilde \
\cong \bigoplus\widetilde{M_i}$.
\end{enumerate}
\end{proposition}

\begin{proof}
A morphism of $K[\X]$-modules $\varphi:\mathscr{A}M \rightarrow
\mathscr{A}N$ induces maps $\varphi_{[p]}:\mathscr{A}M_{[p]}
\rightarrow \mathscr{A}N_{[p]}$ for each $p\in \X$ where on each
open set $U\subseteq S$, $\varphi_{[p],
U}(\frac{[a]}{f})=\frac{\varphi_U([a])}{f}$ for $\frac{[a]}{f}\in
\mathscr{A}M_{[p]}(U)$. It therefore induces a map in the stalk
$\varphi_{[p], \pi(p)}:(\mathscr{A}M_{[p]})_{\pi(p)} \rightarrow
(\mathscr{A}N_{[p]})_{\pi(p)}$ and hence a map
$$\widetilde{\varphi}:\widetilde{\mathscr{A}M} \rightarrow \widetilde{\mathscr{A}N}$$ Hence the
map
$$\widetilde{}:\mathscr{AS}(\X) \rightarrow \mathscr{CS}(\X)$$ defined by
$$\mathscr{A}M\mapsto \widetilde{\mathscr{A}M}$$ is a functor.

It has an inverse functor $\Gamma:\mathscr{CS}(\X) \rightarrow
\mathscr{AS}(\X)$ defined by
$$\mathscr{F}\mapsto \mathscr{AF}(\X)$$ and
$$\Gamma:Hom_{\mathcal{O}_{\X}}(\widetilde{\mathscr{A}M}, \widetilde{\mathscr{A}N}) \rightarrow Hom_{\mathscr{A}K[\X]}(\mathscr{A}M,
\mathscr{A}N)$$ defined by
$$\Gamma(\psi)=\mathscr{A}\psi_{\X}$$ where we use the identification
$\widetilde{\mathscr{A}M}(\X)\cong \mathscr{A}M,
\widetilde{\mathscr{A}N}(\X)\cong \mathscr{A}N$ and
$\psi_{\X}:\widetilde{\mathscr{A}M}(\X) \rightarrow
\widetilde{\mathscr{A}N}(\X)$ is the homomorphism of $\psi$ on $\X$.
Therefore the functor $\widetilde{}$ is fully faithful. The
exactness follows from the exactness of localization.
\end{proof}

\begin{definition}
Let $\X, \Y$ be two semi-topological algebraic sets over a condensed
space $S$. A semi-topological $K$-morphism is a pair
$$(\psi, \phi):(\X, K[\X]) \rightarrow (\Y, K[\Y])$$ which satisfy the following properties:
\begin{enumerate}
\item
$\psi:\X \rightarrow \Y$ is continuous and the following diagram
commutes
$$\xymatrix{\X \ar[rr]^{\psi} \ar[rd] && \Y \ar[ld] \\
& S &}$$

\item $\phi:K[\Y] \rightarrow K[\X]$ is a morphism of almost sheaves such that
$\phi(m^{\Y}_{[\psi(p)]}(U))\subseteq m^{\X}_{[p]}(U)$ for any
open sets $U\subseteq S$
\end{enumerate}
\end{definition}

The following proof was suggested by Li Li.
\begin{proposition}
Suppose that $\X, \Y$ are affine semi-topological algebraic sets
over $S$ and $(\psi, \phi):(\X, K[\X]) \rightarrow (\Y, K[\Y])$ is a
semi-topological $K$-morphism.

\begin{enumerate}
\item
If $N$ is a presheaf of $K[\X]$-modules and $p\in \X$, then
$N\otimes_{K[\X]}K[\X]_{[p]}\cong N_{[p]}$ as presheaves of
$K[\X]$-modules.

\item
If $M$ is a presheaf of $K[\Y]$-module, then
$(\mathscr{A}M\otimes_{\mathscr{A}K[\Y]}\mathscr{A}K[\X])_{[p]}\cong
\mathscr{A}M\otimes_{\mathscr{A}K[\Y]}\mathscr{A}K[\X]_{[p]}$.

\item
Let $N$ be a presheaf of $K[\X]$-modules. Consider $N_{[p]}$ as
a presheaf of $K[\Y]$-modules through $\phi$, then
$((N_{[p]})_{[q]})_{\pi_2(q)}\cong (N_{[p]})_{\pi_1(p)}$ as
presheaves of $K[\Y]_{[p]}$-modules.

\item
If $M$ is an almost sheaf of $K[\Y]$-modules, then
$\varphi^*(\widetilde{\mathscr{A}M})\cong
\widetilde{(\mathscr{A}M\otimes_{\mathscr{A}K[\Y]}
\mathscr{A}K[\X])}\cong
\widetilde{\mathscr{A}(M\otimes_{K[\Y]}K[\X])}$.

\item
If $N$ is a presheaf of $K[\X]$-modules, let $_{\Y}N$ denote $N$
as a presheaf of $K[\Y]$-modules through $\varphi$, then
$\psi_{*}\widetilde{\mathscr{A}N}\cong
\widetilde{(_{\Y}\mathscr{A}N)}$.
\end{enumerate}\label{pullback}
\end{proposition}

\begin{proof}
\begin{enumerate}
\item
Define a presheaf morphism $\varphi:N\times K[\X]_{[p]}
\rightarrow N_{[p]}$ by assigning every open set $U\subseteq S$,
a homomorphism $\varphi_U:(N\times K[\X]_{[p]})(U) \rightarrow
N_{[p]}(U)$ which is defined by
$$(\frac{f}{g}, b)\mapsto \frac{fb}{g}$$ where $\frac{f}{g}\in
K[\X]_{[p]}(U), b\in M(U)$. Since $\varphi_U$ is
$K[\X](U)$-linear, by the universal property of tensor product,
we get a map
$$\varphi'_U:(N\otimes_{K[\X]} K[\X]_{[p]})(U) \rightarrow N_{[p]}(U)$$ and
hence a morphism $$\varphi':N\otimes_{K[\X]} K[\X]_{[p]}
\rightarrow N_{[p]}$$ of $K[\X]_{[p]}$-modules. By
\cite[Proposition 3.5]{AM}, $\varphi'_U$ is an isomorphism of
$K[\X]_{[p]}(U)$-modules for all $U$ and hence $\varphi'$ is a
presheaf isomorphism of $K[\X]_{[p]}$-modules.

\item
By (1),
$(\mathscr{A}M\otimes_{\mathscr{A}K[\Y]}\mathscr{A}K[\X])_{[p]}\cong
(\mathscr{A}M\otimes_{\mathscr{A}K[\Y]}\mathscr{A}K[\X])\otimes_{\mathscr{A}K[\X]}\mathscr{A}K[\X]_{[p]}\cong
\mathscr{A}M\otimes_{\mathscr{A}K[\Y]}(\mathscr{A}K[\X]\otimes_{\mathscr{A}K[\X]}\mathscr{A}K[\X]_{[p]})\cong
\mathscr{A}M\otimes_{\mathscr{A}K[\Y]}\mathscr{A}K[\X]_{[p]}$

\item
Define $\psi:((N_{[p]})_{[q]})_{\pi_2(q)}\rightarrow
(N_{[p]})_{\pi_1(p)}$ by
$$\frac{\frac{a}{f}}{g} \mapsto \frac{a}{\phi(g)f}$$
for $\frac{a}{f}\in N_{[p]}(U)$ and $g\in m^{\Y}_{[q]}(U)$ for
some $U\subseteq S$ a neighborhood of $q$. This map is clearly
an isomorphism.

\item
$\varphi^*(\widetilde{\mathscr{A}M})_p=(\varphi^{-1}\widetilde{\mathscr{A}M}\otimes_{\varphi^{-1}\mathcal{O}_{\Y}}\mathcal{O}_{\X})_p
\cong
(\mathscr{A}M_{[q]})_{\pi(q)}\otimes_{(\mathscr{A}K[\Y]_{[q]})_{\pi(q)}}(\mathscr{A}K[\X]_{[p]})_{\pi(p)}\cong
(\mathscr{A}M_{[q]})_{\pi(q)}\otimes_{(\mathscr{A}K[\Y]_{[q]})_{\pi(q)}}((\mathscr{A}K[\X]_{[p]})_{[q]})_{\pi(q)}
\cong
((\mathscr{A}M\otimes_{\mathscr{A}K[\Y]}\mathscr{A}K[\X]_{[p]})_{[q]})_{\pi(q)}\cong
(((\mathscr{A}M\otimes_{\mathscr{A}K[\Y]}\mathscr{A}K[\X])_{[p]})_{[q]})_{\pi(q)}\cong
((\mathscr{A}M\otimes_{\mathscr{A}K[\Y]}\mathscr{A}K[\X])_{[p]})_{\pi(p)}\cong
(\widetilde{\mathscr{A}M\otimes_{\mathscr{A}K[\Y]}\mathscr{A}K[\X]})_p$
for each $p\in \X$.

\item
$(\psi_*\widetilde{\mathscr{A}N})(\Y)=\widetilde{\mathscr{A}N}(\psi^{-1}(\Y))=\widetilde{\mathscr{A}N}(\X)\cong
_{\Y}\mathscr{A}N=\widetilde{(_{\Y}\mathscr{A}N)}(\Y)$, hence
$\psi_*\widetilde{\mathscr{A}N}\cong _{\Y}\mathscr{A}N$.
\end{enumerate}
\end{proof}

\begin{definition}
Let $\X$ be a semi-topological over a condensed space $S$, and
$\mathscr{F}$ a sheaf on $\X$. We say that $\alpha\in
\mathscr{F}(U)$ is almost zero, denoted by $\alpha=_a0$, if there is
a set $A\subseteq S$ of measure zero such that
$\alpha|_{U-\pi^{-1}(A)}=0$. And we write $\beta\in_a
\mathscr{F}(U)$ if there is a set $B\subseteq S$ of measure zero
such that $\beta\in \mathscr{F}(U-\pi^{-1}(B))$.
\end{definition}

\begin{proposition}
Let $\mathscr{F}$ be a quasi-coherent sheaf on an affine
semi-topological algebraic set $\X$ over $S$. Then the restriction
map from $\mathscr{F}(U)$ to $\mathscr{F}(U-A)$ where $A$ is a set
of measure zero is an isomorphism.
\end{proposition}

\begin{proof}
Suppose that $\X$ is covered by $U_i$ and $\mathscr{F}|_{U_i}\cong
\widetilde{\mathscr{AN}_i}$. Let $A$ be a set of measure zero. Since
$U-A=\cup_i(U\cap U_i-A)$, $\mathscr{F}(U\cap
U_i-A)=\widetilde{\mathscr{AN}}_i(U\cap U_i-A)$. From the following
commutative diagram
$$\xymatrix{ \mathscr{F}(U\cap U_i) \ar[r]^{\cong} \ar[d] &
\widetilde{\mathscr{AN}}_i(U\cap U_i) \ar[d]^{\cong}\\
\mathscr{F}(U\cap U_i-A) \ar[r]^{\cong} &
\widetilde{\mathscr{AN}}_i(U\cap U_i-A)}$$ we see that the
restriction from $\mathscr{F}(U\cap U_i)$ to $\mathscr{F}(U\cap
U_i-A)$ is an isomorphism. And these isomorphisms patch an
isomorphism between $\mathscr{F}(U)$ and $\mathscr{F}(U-A)$.
\end{proof}

\begin{remark}
\begin{enumerate}
\item If $M$ is a presheaf of $K[\X]$-modules generated by global
sections, then $\mathscr{A}M$ is a finitely generated
$\mathscr{A}K[\X]$-module.

\item $\mathscr{AA}M\cong \mathscr{A}M$.

\end{enumerate}
\end{remark}

We have develop enough machinery such that the proof of the
following result is similar to \cite[Lemma 5.3]{Hart}.

\begin{lemma}\label{extending sections}
Let $\X$ be an affine condensed scheme over $S$, $f\in K[\X](S)$
condensed, and let $\mathscr{F}$ be a quasi-coherent sheaf on $\X$.

\begin{enumerate}
\item If $\alpha\in \Gamma(\X, \mathscr{F})$ whose restriction to $D(f)$
is 0, then for some $n>0$, $f^n\alpha=_a 0$.

\item Suppose that $\alpha\in \mathscr{F}(D(f))$, then for some $n>0$, there is an almost global section $\beta\in_a \mathscr{F}(\X)$
such that $\beta|_{D(f)}=f^n\alpha$.
\end{enumerate}
\end{lemma}

\begin{proposition}
Suppose that $\X$ is an affine semi-topological algebraic set over
$S$. Then a bivariant sheaf $\mathscr{F}$ of
$\mathcal{O}_{\X}$-modules is quasi-coherent if and only if for
every affine open subset $U\subseteq \X$, there is an almost sheaf
of $K[U]$-modules $M$ such that $\mathscr{F}|_U\cong
\widetilde{\mathscr{A}M}$. Furthermore, $\mathscr{F}$ is bivariant
coherent if and only the same is true, with the extra condition that
$M$ is an almost sheaf of finitely generated $K[U]$-modules.
\end{proposition}

\begin{proof}
As in the proof above, we may cover $\X$ by finitely many $D(g_i)$
where $g_i\in \mathscr{K}[\X](S)$ such that
$\mathscr{F}|_{D(g_i)}\cong \widetilde{\mathscr{AN}}_i$. By Lemma
above, $\mathscr{F}|_{D(g_i)}\cong \mathscr{AM}_{g_i}$, hence we get
an isomorphism $\varphi_i:\mathscr{AM}_{g_i} \rightarrow
\mathscr{AN}_i$. Since $\widetilde{\mathscr{AM}_{g_i}}\cong
\widetilde{\mathscr{AM}}|_{D(g_i)}$, we get an isomorphism
$$\widetilde{\varphi}_i:\widetilde{\mathscr{AM}} \rightarrow
\mathscr{F}|_{D(g_i)}$$ These isomorphisms patch together to give us an isomorphism
$$\widetilde{\varphi}:\widetilde{\mathscr{AM}} \rightarrow
\mathscr{F}$$
\end{proof}

\begin{corollary}
Suppose that $\X$ is an affine condensed scheme over $S$. The
functor $M\mapsto \widetilde{M}$ is an equivalence of categories
between the category of sheaves of $K[\X]$-modules and the category
of bivariant quasi-coherent sheaves. Its inverse is given by
$\mathscr{F}\mapsto \mathscr{F}(\X)$. The same functors gives an
equivalence of categories between the category of sheaves of
finitely generated $K[\X]$-modules and the category of bivariant
coherent sheaves.
\end{corollary}

\bibliographystyle{amsplain}

\end{document}